\documentclass[11pt]{amsart}
\usepackage{amssymb}
\usepackage{amsmath}
\usepackage{amsfonts}
\usepackage[usenames]{color}
\usepackage{array}
\usepackage{color}

\makeatletter
 
 \@addtoreset{equation}{section}
\makeatother

\newtheorem{theorem}{Theorem}[section]
\newtheorem{lemma}[theorem]{Lemma}

\theoremstyle{definition}

\newtheorem{proposition}[theorem]{Proposition}

\newtheorem{corollary}[theorem]{Corollary}

\theoremstyle{remark}
\newtheorem{remark}[theorem]{Remark}

\numberwithin{equation}{section}

\begin{document}

\author{Andrei Pavelescu}
\address{Department of Mathematics, Oklahoma State University, Stillwater, OK 74078-1058, USA}
\email{andrei.pavelescu@okstate.edu}

\date{\today}
\thanks{The author was partially supported by NSF grant
DMS 1001962.}

\title{Derangements in cosets of primitive permutation groups}
 
\maketitle

\begin{abstract}  Motivated by questions arising in connection with branched coverings of connected smooth projective curves over finite fields, we study the proportion of fixed point free elements (derangements) in cosets of normal subgroups of primitive permutations groups. Using the Aschbacher--O'Nan--Scott theorem for primitive groups to partition the problem, we provide complete answers for affine groups and groups which contain a regular normal nonabelian subgroup.

\end{abstract}

\section{Introduction}

\hspace{.25in}The study of the fixed points of permutations has a long history, starting with a probability theorem of Montmort \cite{M} who proved that the average number of permutations of size $n$ with $k$ fixed points tends to $1/(ek!)$. Further connections with probability theory come in relation with card shuffling.\\

In an algebraic context, let $A$ be a transitive permutation group acting on a set $\Omega$ with $n$ elements. Let $S_0$ denote the set of fixed point free elements of $G$. By a classical result of C. Jordan \cite{Jo}, $S_0$ is nonempty. Motivated by number theoretic applications such as the number field sieve, H.W. Lenstra, Jr. \cite{BLP} asked for a lower bound for
\[s_0:=\frac{|S_0|}{|A|}.\] Cameron and Cohen \cite{CC} proved that $s_0\ge 1/n$ with equality if and only if $A$ is a Frobenius group of order $n(n-1)$, with $n$ a prime power.  Guralnick and Wan \cite{GW} proved that if $s_0>1/n$, then the next bound is $s_0=2/n$, with equality for a Frobenius group of order $n(n-1)/2$ with $n$ an odd prime power, $\mathbb{Z}/3\mathbb{Z}$ or $A_5$.\\\\
The question of studying $s_0$ arises in a more arithmetic setting.\\

Let $\mathbb{F}_q$ be a finite field of characteristic $p$ and let $f(T)\in \mathbb{F}_q[T]$ be a polynomial of degree $n>1$ which is not a polynomial in $T^p$. S. Chowla \cite{Ch} asked for an estimation of $V_f:=|f(\mathbb{F}_q)|$. A result of Birch and Swinnerton-Dyer \cite{BS} shows that, provided the Galois group of $f(T)-t=0$ over $\overline{\mathbb{F}}_q(t)$ is $S_n$, then:
\[V_f= \left(\sum_{k=1}^n\frac{(-1)^{k-1}}{k!} \right)q+ O(\sqrt{q}),\]
where the constant term in the error does not depend on $f$, only on $n$.  Unless $f$ is a permutation polynomial ($V_f=q$), $V_f<q$ and a known elementary upper bound is \[V_f\le q-\frac{q-1}{n}.\]
When interested in asymptotic upper bounds, it turns out they do depend on $f$.\\

Let $A$ be the Galois group of $f(T)-t=0$ over $\mathbb{F}_q(t)$ and let $G$ be the Galois group of $f(T)-t=0$ over $\overline{\mathbb{F}}_q(t)$. Both groups act transitively on the set of $n$ roots of $f(T)-t=0$. Furthermore, the geometric monodromy group $G$ is a normal subgroup of the arithmetic monodromy group $A$, with $A/G$ cyclic generated by $xG$. If $S_0$ denotes the set of fixed point free elements in $xG$, then the Cebotarev density theorem for function fields yields: 
\[V_f= (1-\frac{|S_0|}{|G|})q+O(\sqrt{q}),\] with the constant error term depending only on $n$. Therefore the problem is reduced to the understanding of $s_0=|S_0|/|G|$.\\

Unless $f$ is an exceptional polynomial (it induces bijections in arbitrarily large degree extensions of $\mathbb{F}_q$), $s_0>0$ with the next bound $s_0= 1/n$ holding if $A=G$ is a Frobenius group of order $n(n-1)$ with $n$ a prime power (Lenstra).  Not surprisingly, the next bound for $s_0$ is $2/n$ as given in:\\

\begin{theorem}[Guralnick--Wan]
Let $f(T)$ be a polynomial over $\mathbb{F}_q$ of degree $n>6$ which is not a polynomial in $T^p$. If $s_0>1/n$, then $s_0\ge 2/n$ with equality holding iff $A=G$ is a Frobenius group of order $n(n-1)/2$ with n a prime power. In particular, $V_f \le (1-2/n)q+O_n(\sqrt{q})$ unless $f$ is exceptional or $A=G$ is a Frobenius group of order $n(n-1)$.
\end{theorem}
 The proof of this theorem uses the classification of finite simple groups.\\
 
 If the degree of $f$ is not divisible by the characteristic of $\mathbb{F}_q$,  which is the same as saying that $f$, seen as a morphism from $\mathbb{P}^1$ to $\mathbb{P}^1$, has tame ramification at $\infty$, then $s_0>1/6$ whenever $s_0>0$ \cite{GW} . Thus either $f$ is bijective or $V_f \le (5/6)q+O_n(\sqrt{q})$.
 
 If all ramification is tame, then Guralnick and Wan \cite{GW} proved that $s_0>0$ implies $s_0\ge 16/63$ and the bound is the best possible.\\
 
Guralnick and Wan \cite{GW} generalized these results  to branched coverings of smooth projective curves defined over a finite field. Reducing to the case where the covering is indecomposable (the corresponding arithmetic monodromy group is primitive), the authors concluded:\\

\begin{theorem}[Guralnick--Wan]
Let $\alpha: \Omega\rightarrow Y$ be a separable branched covering of degree $n$ with $\Omega,Y,\alpha$ defined over $\mathbb{F}_q$. Assume that one of the branch points is totally ramified and is $\mathbb{F}_q$-rational. Let $p$ be the characterisitic of $\mathbb{F}_q$. Let $A$ be the arithmetic monodromy group of the covering and $G$ the geometric monodromy group. Then one of the following holds:
\begin{enumerate}
\item[(a)] $r_2=0$, $s_0=0$ and the covering is exceptional;
\item[(b)] $r_2=1$, $s_0=1/n$ and $A=G$ is Frobenius of order $n(n-1)$ with $n$ a prime or $p^a$;
\item[(c)] $r_2=2$, $s_0=2/n$ and $A=G$ is Frobenius of order $n(n-1)/2$ with $n$ an odd prime or $p^a$ (with $p>2$);
\item[(d)] $s_0>2/n$; or
\item[(e)] $n\le 6$, $A=G$ and $1/n\le s_0 \le 2/n$ or $n=4$, $|A/G|=2$ and $s_0= 2/4$. 

\end{enumerate}\vspace{.1in}
\end{theorem}

In the same paper, the authors commented that ``...there should be a version of the previous result without the assumption that we are dealing with monodromy groups of polynomials (or more generally coverings with a totally ramified rational point)''. In this paper, we study this situation and provide answers for the affine case and the regular nonabelian normal subgroup case.

\section{Machinery}\vspace{.1in}

Let $A$ be a permutation group acting on a set $\Omega$ which has $n$ elements. Denote by $\mu(\ge 2)$ the minimal number of elements moved by a nonidentity element of $A$. For an element $x$ of $A$, let $Fix(x)$ denote the set of elements of $\Omega$ that are fixed by $x$.
\vspace{.1in}

\begin{lemma} The number of orbits of $A$ acting on $\Omega$ is less or equal to $n-\frac{\mu}{2}$.
\label{1}
\end{lemma}\vspace{.1in}
\begin{proof} If $r$ is the number of orbits of $\Omega $ under the action of $A$,
then, by Burnside's Lemma
\[r|A|=\sum_{\Omega\in A}|Fix(x)|=n+\sum_{x\ne 1}|Fix(x)|\le n+ (|A|-1)(n-\mu)= |A|n-(|A|-1)\mu\Rightarrow\]
\[\Rightarrow r\le n-\frac{|A|-1}{|A|}\mu \le n-\frac{\mu}{2}.\]
\end{proof}
\begin{remark}Notice that the above inequality is strict unless $|A|=2$.
\end{remark}\vspace{.1in}
Let $G$ be a normal subgroup of $A$ such that $A/G$ is cyclic, generated by $x$.\vspace{.1in}
\begin{lemma} Let $r=r(\Omega)$ be the number of common $(A,G)$--orbits on $\Omega$. Then \[\frac{1}{|G|}\sum_{g\in xG}|Fix(g)|=r.\]\label{2} \end{lemma}\vspace{.1in}
\begin{proof} Without loss of generality, one may assume $A$ is transitive.
(The orbits of the $A$-action form a partition of $\Omega$ which has as a
subpartition the orbits of the $G$-action).\\\\
\textit{Claim:} $G$ is transitive iff there exists $g\in G$
 such that $xg$ has a fixed point.\\

 Proof of claim: "$\Rightarrow$" Let $\alpha\in \Omega$. Since $G$ is transitive, $\exists g\in G$ such that $g(\alpha)= x^{-1}(\alpha)\Rightarrow xg(\alpha)=
 \alpha$.\\ 

 "$\Leftarrow$" Let $g\in G$, $\alpha\in \Omega$ such that $xg(\alpha)=\alpha\Rightarrow g^{-1}x^{-1}(\alpha)= \alpha \Rightarrow xg^{-1}x^{-1}(\alpha)=x(\alpha) \Rightarrow g_1(\alpha)=x(\alpha)$ for some $g_1\in G$, as $G$ is a normal subgroup. If $\beta$ is an arbitrary element of $\Omega$,
  since $A$ is transitive,
 there exists $h\in A$ such that $\beta=h(\alpha)$. Under the current assumptions on $A$ and $G$, there exist $t\in \mathbb{N}$ and $g_2\in G$ such that
  $h= x^tg_2$.
 It then follows that \[ \beta=x^tg_2(\alpha)= x^{t-1}xg_2x^{-1}x(\alpha)= x^{t-1}g_3(\alpha)=...=x^{t-i}g_{i+2}(\alpha)=...=g_{t+2}(\alpha)
 \] where inductively $g_{i+1}:= xg_ix^{-1}g_1\in G$, since $G$ is normal in $A$. Since $\beta$ was arbitrary, it follows that $G$ is transitive.\\

 By the claim, if $G$ is not transitive, both sides of the equation are 0. So we assume $G$ is transitive ($r$=1). Set \[Y=\{(xg,\omega)\in xG\times \Omega| xg(\omega) = \omega\},\] nonempty by the claim.
 Let $A_{\omega}$ and $G_{\omega}$ denote the corresponding point stabilizers. If $xg(\omega)= \omega,$ then $A_{\omega}\cap xG = xgG_{\omega}$,
  thus $|A_{\omega}\cap xG| = |G_{\omega}|.$ One has
  \[\sum_{g\in xG}|Fix(g)|=|Y|= \sum_{\omega\in \Omega}|A_{\omega}\cap xG|= \sum_{\omega\in \Omega}|G_{\omega}| =\sum_{\omega\in \Omega}\frac{|G|}{|\Omega|} = |G|,  \]
   since $G$ was assumed transitive.\end{proof}
\begin{remark} Rephrasing, the above result states that the average number of fixed points in a generating coset equals the number of common orbits. When looking at the proportion of fixed points, this result provides us with several combinatorial approaches. If $G$ is transitive, this implies that the average number of fixed points is 1.
\end{remark}\vspace{.1in}

 From this point on, unless
   otherwise specified, $A$ and $G$ are assumed transitive. For all $0\le i \le n$, define $S_i:=\{g\in xG: |Fix(g)|=i\}.$ Let $s_i:=|S_i|/|G|$. Let $r_k$ denote the number of common $(A,G)$--orbits of the component-wise actions on $\Omega^{(k)}$, the $k$-fold cartesian product with all diagonals removed. \vspace{.1in}
\begin{lemma} The following are equivalent:
\begin{enumerate}
\item[(a)] $r_2=0$;
\item[(b)] $s_0=0$;
\item[(c)] every element in the coset $xG$ fixes a unique point;
\item[(d)] every element in the coset $xG$ fixes at most one point;
\item[(e)] every element in the coset $xG$ fixes at least one
point.
\end{enumerate}\label{3}
\end{lemma}\vspace{.1in}

\begin{proof} Since $A$ and $G$ are transitive, by Lemma \ref{2} it follows that (c),(d) and (e) are
equivalent. Also, by the definition of $s_0$, (b) and (e) are
equivalent. For $a,b\in \Omega$, $a\ne b$ we have $A(a,b)=G(a,b)$ if and
only if there exists $g\in G$ such that $xg\in A_a\cap A_b$. Thus
$r_2\ne 0$ is equivalent to some element in $xG$ fixing at least
 two points; therefore (a) is equivalent to (d).\\
 \end{proof}
 A triple $(A,G,\Omega)$ with the above properties is called \textit{exceptional}. The name is consistent with the situation where $A$ and $G$ are respectively the arithmetic and geometric monodromy groups of an exceptional polynomial. In the following we focus on the non-exceptional case, namely $r_2\ge 1$.
\section{Combinatorics}
From the definitions, one has $s_n=1/|G|$ and $s_{n-1}=s_{n-2}=...=s_{n-\mu+1}$, where $\mu$ denotes the minimal degree of $A$. When one looks at the relation between $s_0$ and $r_2$ under the new assumptions, one proves
\begin{lemma} Assuming $r_2\ge 1$, 
\[s_0\ge
\frac{r_2}{n}+\frac{(n-2)r_2-r_3}{n(n-\mu)}.\] 
\end{lemma}
\begin{proof} Since $A$ and
$G$ both act transitively, it follows that $r_1=1$. Furthermore,
\begin{equation}
s_0+s_1+s_2+...+s_n=1. \label{e1}
\end{equation}\\
By Lemma \ref{2} applied to $(A,G,\Omega^{(k)})$, for $1\le k\le n$, we
get \[r_k=\frac{1}{|G|}\sum_{g\in xG}|Fix(g)|=
\frac{1}{|G|}\sum_{i=0}^n\sum_{g\in S_i}|Fix(g)|=
\frac{1}{|G|}\sum_{i=k}^n
|S_i|P_i^k=\sum_{i=k}^n\frac{|S_i|}{|G|}\binom{i}{k}\cdot k!, \]
which yields
\begin{equation}
\sum_{i=k}^n\binom{i}{k}s_i = \frac{r_k}{k!}, \hspace{.6 in} 1\le k \le n.
\label{e2}
\end{equation}\\
\begin{remark}The sums actually go up to $n-\mu$ as $s_n=s_{n-1}=...=s_{n-\mu+1}=0$.
\end{remark}\vspace{.1in}
Subtracting \ref{e1} from the first equation of \ref{e2}, one gets
\begin{equation}
s_0=\sum_{i=2}^n\binom{i-1}{1}s_i. \label{e3}
\end{equation}\\
By multiplying \ref{e3} by $\frac{n}{2}$ and subtracting the third
equation of \ref{e2}, we get
\begin{equation}
\frac{ns_0}{2}-\frac{r_2}{2}=\sum_{i=2}^{n-\mu}\frac{(n-i)(i-1)}{2}s_i\ge
0. \label{e4}
\end{equation} \\ The last formula immediately implies that $s_0\ge r_2/n$, with equality if and only if
$s_2=s_3=...=0$. Since $r_2\ge 1$, there exists $(a,b)\in \Omega^{(2)}$ such that $A(a,b)=G(a,b)$. But then, there exists $g\in G$ such that
$x^{-1}(a,b)=g(a,b) \Rightarrow xg(a,b)=(a,b)\Rightarrow xg\in xG\cap A_{a,b} $; this can only happen if $x=g^{-1}$ as $s_2=s_3=...=0$, thus $A=G$  is a Frobenius group. Unless $|A|=n(n-1)$ or $|A|=n(n-1)/2$, $s_0=(n-1)/|G|>2/n$. The stabilizer of a point acts as fixed point free automorphisms of the regular normal subgroup $N$. Thus, by considering nontrivial conjugacy classes, it follows that $N$ is a $p$-elementary group with $p$ prime. Thus $n=p^a$ with $p$ odd if $|A|=n(n-1)/2$.
\\\\Multiplying the second equation of \ref{e2} by $\frac{n-2}{3}$ and
subtracting the third equation of \ref{e2},  it follows that
\begin{equation}
\frac{(n-2)r_2-r_3}{6}=\sum_{i=2}^{n-\mu}\frac{(n-i)i(i-1)}{3!}s_i\ge0.
\label{e5}
\end{equation}\\
Finally, by multiplying \ref{e4} by $\frac{n-\mu}{3}$ and
subtracting \ref{e5}, one gets
\[\frac{n-\mu}{3}(\frac{ns_0}{2}-\frac{r_2}{2})=\sum_{i=2}^{n-\mu-1}\frac{(n-i)(n-\mu-i)(i-1)}{3!}s_i\ge
0,\]thus
\begin{equation}
s_0\ge \frac{r_2}{n}+\frac{(n-2)r_2-r_3}{n(n-\mu)}.
\label{e6}
\end{equation}
\end{proof}
By \ref{e5}, if $r_2\ge 2$, then $s_0\ge \frac{2}{n}$.\\

In a similiar setting, Guralnick and Wan \cite{GW}(Lemma 3.5 and its corollary) proved that it suffices to study the case where $A$ is primitive. For the rest of this section, we shall assume $r_2=1$ and $A$
primitive. Note that this implies $G$ is transitive since otherwise the orbits of $G$ would constitute imprimitivity blocks for $A$. The bound in \ref{e6} reduces to
\begin{equation}
s_0\ge \frac{1}{n}+\frac{n-(r_3+2)}{n(n-\mu)}, \label{e7}
\end{equation}which immediately implies the following lemma.\vspace{.1in}

\begin{lemma} If $r_3+2<\mu$, then $s_0>\frac{2}{n}$.
\label{4}
\end{lemma}
Let $(a,b)$ be a representative of the common ($A,G$)--orbit on $\Omega^{(2)}$. Let $A_{a,b}$ be the stabilizer of $a$ and $b$ acting on $\Omega$ and let $r$ denote the number of orbits of this action.
\vspace{.1in}\begin{proposition} $r_3+2\le r\le n-\frac{\mu}{2}$.\label{5}
\end{proposition}

\begin{proof} First notice that if $(a,b,c)\in \Omega^{(3)}$ such that $A(a,b,c)=G(a,b,c)$, then  $A(a,b)=G(a,b)$ and thus for every
 $i=1,2,...,r_3$, there exists $c_i\in \Omega$ such that $(a,b,c_i)\in O_i$, where $O_1,O_2,...,O_{r_3}$ are the common ($A,G$) orbits on $\Omega^{(3)}$.
  Denote by $\{o_1,...o_r\}$ the collection of $A_{a,b}$ orbits and define a set map from $\varphi:\{O_1,O_2,...,O_{r_3}\}\rightarrow \{o_1,...o_r\} $ by
   $\varphi(A(a,b,c_i)) = A_{a,b}c_i$. Then \[\varphi(A(a,b,c_i))=\varphi(A(a,b,c_j)) \Leftrightarrow A_{a,b}c_i= A_{a,b}c_j\Leftrightarrow \exists g\in A,
    g(a,b,c_i)=(a,b,c_j),\] which is to say, $\varphi$ is a well-defined injection. Since $\{a\},\{b\} \notin Im(\varphi)$, it follows that
    $r_3+2\le r\le n-\frac{\mu}{2}$, by Lemma \ref{1}.
\end{proof}
The above result, \ref{4} and the remark following Lemma \ref{1} imply the following lemma.\vspace{.1in}
\begin{lemma}\hspace{4in}

\begin{enumerate}
\item[a)] If $\mu>\frac{2n}{3}$, then $s_0>\frac{2}{n}$;
\item[b)] If $ \mu = \frac{2n}{3}$, then $s_0>\frac{2}{n}$ unless $A_{a,b}$ is a subgroup of order 2.
\end{enumerate}
\label{6}
\end{lemma}

\begin{flushleft} At this point, we need more information about $\mu$. It turns out that if the group $A$ is affine, the extra geometric structure is sufficient to fully classify all possibilities in this case. We shall handle this in the following section. We conclude this section with the following useful remark:
\end{flushleft}
\vspace*{.1in}

Under the assumption that $r_2=1$, we can derive an upper bound for $s_0$:
\[s_0\le \sum_{i=0}^ns_i\bigg (\binom{i}{0}-\binom{i}{1} +\binom{i}{2}\bigg)= 1-r_1+\frac{r_2}{2}=\frac{1}{2}.\] 
\begin{remark}
The equality in the above inequality holds if and only if $s_3=s_4=...=s_n=0$.
\label{200}
\end{remark}
\section{The Affine Case}

From this point on we are going to assume that $A$ is affine, acting on a $d$-dimensional $\mathbb{F}_q$-vector space $V$, with $n=q^d$.
 One can identify $V$ as a subgroup of $A$ (as translations) and $A=VA_0$, with  $A_0$ the isotropy group (the point stabilizer of 0) .
 
 \vspace*{.1in}
The following result follows from Lang's Theorem (\cite{GW}, Lemma 2.3)

\begin{lemma}\hspace{4in}
\begin{enumerate}

\item[a)] If $d>1$, then $\mu\ge \frac{(q-1)n}{q}$.

\item[b)] If $d=1$ and $q$ is prime, then $\mu=q-1$.

\item[c)] If $d=1$ and $q=q_0^e$ with $e$ prime and minimal, then $\mu \ge q-q_0$.
\end{enumerate}
\label{7} \end{lemma}
\begin{flushleft}By Lemma \ref{6} and Lemma \ref{7} it follows
that $s_0>2/n$ unless $q\le 3$ or  $n=4$ or 9.\\
\end{flushleft}
If  $n=4$, then $A=A_4$ or $S_4$. Since the Klein 4--group is
not 2--transitive and $S_4$ and $A_4$ are, this implies that\vspace{.1in}

1. $(A,G,V)$ is exceptional or

2. $A=S_4$, $G=A_4$, $s_0=\frac{1}{2}$ or

3. $A=G=A_4$, $A$ is
Frobenius and $s_0=\frac{1}{4}$.\\\\
If $n=9$, then either $A=G$ is Frobenius, $(A,G,V)$ is exceptional,
or $s_0\ge
\frac{1}{3}$.\\\\

If $q=3$, $n=3^d>9$ by Lemmas \ref{6} and \ref{7} it follows that
either $A$ is Frobenius, $s_0>2/n$ or $x$ acts as a
reflection and $A_{0,v}$ has order 2 for any nonzero $v$ fixed by
$x$. Let $W$ denote the hyperplane fixed by $x$. For any such $v$, $x^tg(0,v)= x^tgx^{-t}(0,v)\in G(0,v)$ shows
that $A(0,v)=G(0,v)$. As $r_2=1$, for any two nonzero distinct elements $w,w'\in W$, $(0,w)$ and $(0,w')$ are contained in the common $(A,G)$--orbit. It follows that all nonzero elements of $W$ are contained in the same $A_0$--orbit. In particular $a(v)=w$ for $a\in A_0$ and $0\ne w\in W$ shows that $a^{-1}xa(v)=v$ and $a^{-1}xa(0)=0$.
Since $A_{0,v}=\{1,x\}$, the centralizer of $x$ acts transitively on
all nonzero elements of $W$. Let $u$ be a vector in the eigenspace of $-1$. As $A_0$ is irreducible on $V$
(otherwise the nontrivial $A_0$--invariant subspace of V would constitute an imprimitivity block for A), $A_0u=A_0v$. For some $a\in A_0$, $a(u)=v$, thus $u=a^{-1}xa(u)$, so this means some reflection
$x'\ne x$ centralizes $u$. As $d>2$, $x$ and $x'$ both fix some nonzero
vector $w$ in $W$. Then $A_{0,w}$ has order greater than 2 and so
does $A_{0,v}$ as $w$ and $v$ are in the same $A_0$--orbit. But this
is a contradiction.\\

In the case $n=2^d>4$, we may assume $x$ fixes 0. If $xG_0$ does not
contain a transvection (unipotent element fixing a hyperplane), then
$\mu\ge 3n/4$ and thus, by Lemma \ref{6}, $s_0>\frac{2}{n}$.
So, without loss of generality, we may assume $x$ is a transvection.
As above, if $W$ is the fixed hyperplane of $x$, as $r_2=1$, all the
nonzero vectors of $W$ are in the same $A_0$--orbit.\\

Let $H$ be the subgroup of $A_0$ generated by transvections. As all
the nonzero elements of $W$ are in the same $A_0$--orbit, for each
$w\in W\backslash\{0\}$, there is a transvection $\tau_w$ centered
at $w$. This leaves $W$ as the only candidate for a nontrivial
invariant subspace. Since $H$ is normal in $A_0$,  this implies
$A_0$ leaves $W$ invariant which is a contradiction.  It follows by
\cite{Mc} that the only irreducible subgroups of $GL_d(V)$ for which a
single orbit contains all nonzero vectors in a hyperplane are
$SL_d(V)$ or $Sp_d(V)$, with $d$ even in the last case.\\

In the first case, $A_0=SL_d(2)$ is 2--transitive on
$V\backslash\{0\}$, so $A$ is 3--transitive. Thus $r_3+2\le 1+2<4\le
\frac{n}{2}=\mu$. In the second case, $A_{0,v}$ has 3 orbits of
nonzero vectors so, by Proposition \ref{5}, $r_2+2\le 4<8\le
\frac{n}{2}=\mu$, as $n\ge
16$. Lemma \ref{4} shows that in both cases
$s_0>2/n$.

\begin{flushleft} We summarize these results in the following theorem.
\end{flushleft}

\begin{theorem} Let $A$ and $G$ be as above. Then one of the following holds:
\begin{enumerate}
\item[(a)] $r_2=0$,  $s_0=0$ and $(A,G,V)$ is exceptional; 
\item[(b)] $r_2=1$,  $s_0=1/n$ and $A=G$ is Frobenius of order $n(n-1)$ with $n=p^d$.
\item[(c)] $r_2=2$, $s_0=2/n$ and $A=G$ is Frobenius of order $n(n-1)/2$ with $n=p^d$ and $p$ is odd.
\item[(d)] $s_0>2/n$;
\item[(e)] $A=S_4$, $G=A_4$ and $s_0=2/4$.
\end{enumerate}
\label{8}
\end{theorem}

\section{Regular Normal Nonabelian Subgroup}
In this section, we consider the case where the socle of the $A$ is regular, but not abelian. The following lemmas will provide a complete answer. \vspace{.1in}
\begin{lemma}Let $A$, $G$ be transitive subgroups with $A/G$ cyclic generated by $xG$.
Suppose that for any $xg\in xG$, we know that either $xg$ is a
derangement or $xg$ has at least $s > 1$ fixed points.   Then the
proportion of derangements in $xG$ is at least  $1 - 1/s \ge 1/2$.
\label{9}
\end{lemma}

\begin{proof} Let $d$ denote the number of derangements in $xG$. By Lemma \ref{2}, since both $A$ and $G$ act
transitively on $\Omega$, we have \[1= \frac{1}{|G|}\sum_{xg\in
xG}|Fix(xg)|\ge \frac{s(|G|-d)}{|G|} ,\] which, by solving for
$d/|G|$, yields the required inequality.
\end{proof}
\begin{remark}
In the above inequality, equality holds precisely when $s_0=s_2=1/2$.
\end{remark}
\begin{lemma} Let $A$ be a finite group acting transitively on a set $\Omega$.
If $N$ is a normal regular subgroup of $A$, then if $g\in A$ has a
fixed point, the number of fixed points is $|C_N(g)|$.
\label{10}
\end{lemma}

\begin{proof} Let $\alpha\in Fix(g)$ be a fixed point of $g$. Let $f:C_N(g)\rightarrow
Fix(g)$ be defined by $f(x)=x(\alpha)$. \\\\
If $x\in C_N(g)\Rightarrow gx=xg\Rightarrow
g(x(\alpha))=x(g(\alpha))=x(\alpha)\Rightarrow x(\alpha)\in
Fix(g)$, so $f$ is well-defined.\\\\
 Let $x,y\in C_N(g)$ such that $f(x)= f(y)$, which means
$x(\alpha)=y(\alpha) \Leftrightarrow y^{-1}x(\alpha)=\alpha$.
Since $N$ is a regular subgroup, $x=y$ and thus $f$ is
injective.\\\\
 Let $\beta\in Fix(g)$. As $N$ is regular, and thus transitive,
there exists a (unique) element $x\in N$ such that
$x(\alpha)=\beta$. Since $\alpha, \beta \in Fix(g)$ we
have 
\[x(\alpha)=\beta\Rightarrow xg(\alpha)=g(\beta)\Rightarrow
g^{-1}xg(\alpha)=\beta.\] Since $N$ is normal $g^{-1}xg\in N$ and
as $N$ is regular $g^{-1}xg(\alpha)=\beta= x(\alpha)$ implies
$g^{-1}xg=x$, thus $x\in C_N(g)$ and therefore $f$ is
surjective. It follows that $f$ is a bijection and so
$|C_N(g)|=|Fix(g)|$.
\end{proof}
\begin{proposition} Let $N$ be a finite group and $x\in Aut(N)$ with $C_N(x)=1$ ($x$ fixes only the identity). Then $N=\{g^{-1}g^x:g\in N\}$.
\label{11}
\end{proposition}
\begin{proof}  Let $g,h\in N$ so that $h^{-1}h^x=g^{-1}g^x$. Then $hg^{-1}=(hg^{-1})^x$ and therefore $hg^{-1}=1 \Rightarrow g=h$. Thus $g\to g^{-1}g^x$ is an injective map from a finite set to itself, therefore a bijection.\\
\end{proof}
%Lemma 3.22 {GMS}
\begin{lemma} Let $N$ be a finite group and $x\in Aut(N)$ with $C_N(x)=1$.
Then $N$ is solvable.
\label{12}
\end{lemma}
\begin{proof} By contradiction, assume there exist nonsolvable groups verifying the hypothesis and pick such group $G$ with $|G|$ minimal.  Let $A$ denote the semidirect product $G \rtimes <x>$. Since $h^{-1}xh=h^{-1}xhx^{-1}x=h^{-1}h^{x^{-1}}x$, by Proposition \ref{10} (applied to $x^{-1}$) it follows that the hypothesis is equivalent to $Gx$ being a single conjugacy class in $A$. As the hypothesis holds for the $x$-invariant sections of $G$, one may assume $G$ is characteristically simple. By \cite{FGS}(Lemma 12.1), there exists an involution $t\in G$ with $t^G= t^A$ which implies $A=GC_A(t)$. Then there exists $g\in G$ such that $gx\in C_A(t)$ so $t\in C_G(gx)$ which is therefore nontrivial. As $gx$ and  $x$ are conjugates (via an element of $G$), it follows that $C_G(x)$ is nontrivial, a contradiction.
\end{proof}\vspace{.1in}
Notice that this immediately implies:\vspace{.1in}
\begin{corollary}  If $N$ is a direct product of simple nonabelian groups and $x\in Aut(N)$,
then $|C_N(x)| > 1$.
 \label{13}
\end{corollary}
In the setting of Lemma \ref{9},  any element $g\in A$  can be view as an element of $Aut(N)$ as acting by conjugation.  Moreover, the centralizer of $g$ in $N$ is the same as the set of fixed points of $g$ acting by conjugation on $N$. Assuming $A$ is primitive and not affine, the socle of $A$ is $H\simeq T^m$, with $T$ simple nonabelian. If $H$ is regular (one of the cases deriving from the Aschbacher--O'Nan--Scott Theorem), then\\

\begin{theorem}
Let $A$ be a primitive permutation and $G$ a normal subgroup such that $A/G$ is cyclic. Assume that the socle of $A$ is regular nonabelian. Then $s_0\ge \frac{1}{2}$.
\end{theorem}

On the other hand,  by Remark \ref{200}, if we assume that $r_2=1$, then it follows that $s_0=s_2=1/2$.


\newpage


\begin{thebibliography}{99}
\addcontentsline{toc}{chapter}{\bibname}

\bibitem[BLP]{BLP} J.P. Buhler, H. W. Lenstra and C. Pomerance, \textit{Factoring integers with the number field sieve}, in \textit{The Development of The Number Field Sieve}, Lecture Notes in Mathematics \textbf{1554}, Springer-Verlag, Berlin, 1993. 
\bibitem[BS]{BS} B. J. Birch and H. P. F. Swinnerton-Dyer, \textit{Note on a problem of Chowla}, Acta Arithmetica \textbf{5} (1959), 417-423. 
\bibitem[CC]{CC} P. J. Cameron and A. M. Cohen, \textit{On the number of fixed point free elements in a permutation group}, Annals of Discrete Mathematics \textbf{106/107} (1992), 135-138.
\bibitem[Ch]{Ch} S. Chowla, \textit{The Riemann zeta and allied functions}, Bulletin of the American Mathematical Society \textbf{58} (1952), 287-303.
\bibitem[DM]{DM} J. D. Dixon and Brian Mortimer, \textit{Permutation Groups}, Springer-Verlag New York-Berlin-Heidelberg, 1996 
\bibitem[FGS]{FGS} M. Fried, R. Guralnick and J. Saxl, \textit{Schur covers and Carlitz's conjecture}, Israel Journal of Mathematics \textbf{82} (1993), 157-225.
\bibitem[GW]{GW} R. Guralnick and D. Wan, \textit{Bounds for Fixed Point Free Elements in a Transitive Group and Applications to Curves over Finite Fields},  Israel Journal of Mathematics \textbf{101} (1997), 255-287.
\bibitem[Jo]{Jo}  C. Jordan, \textit{Recherches sur les substitutions}, J. des Math. Pures et Appl. (Liouville) \textbf{17} (1872), 351-367 (Oeuvres, I, no. 52).
\bibitem[M]{M} P.R. de Montmort, Essay d'analyse sur les jeux de hazard, (1708) 1st ed. , (1713)
(2nd ed.). Jacques Quillau, Paris. Reprinted 1980 by Chelsea, New York.
\bibitem[Mc]{Mc} J. McLaughlin, \textit{Some subgroups of $Sl_n(F_2)$}, Illinois Journal of Mathematics \textbf{13} (1969), 108-115.
\bibitem[Wi]{Wi} R. A. Wilson, \textit{The Finite Simple Groups}, Springer London Dordrecht Heidelberg New York, 2009

\end{thebibliography}
\end{document}